\documentclass[leqno,12pt,a4paper]{amsart}

\usepackage{amsmath}
\usepackage{amssymb}
\usepackage{amsthm}
\usepackage{enumitem}
\usepackage{float}
\usepackage[T1]{fontenc}
\usepackage{mathtools}
\usepackage{microtype}
\usepackage{xcolor}

\usepackage[all,cmtip]{xy}

\usepackage[colorlinks,linkcolor=blue,anchorcolor=blue,citecolor=green,backref=page]{hyperref}
\usepackage[alphabetic]{amsrefs}
\usepackage[capitalise]{cleveref}

\theoremstyle{plain}
\newtheorem{thm}{Theorem}[section]
\newtheorem{claim}[thm]{Claim}

\newtheorem{corollary}[thm]{Corollary}

\newtheorem{lemma}[thm]{Lemma}

\newtheorem{theorem}[thm]{Theorem}

\theoremstyle{definition}
\newtheorem{definition}[thm]{Definition}

\newtheorem*{notation*}{Notation and Terminology}

\newtheorem{remark}[thm]{Remark}

\theoremstyle{remark}

\numberwithin{equation}{section}

\DeclareMathOperator{\Aut}{Aut}

\DeclareMathOperator{\GL}{GL}

\newcommand{\id}{\operatorname{id}}

\newcommand{\Ker}{\operatorname{Ker}}

\newcommand{\PGL}{\operatorname{PGL}}

\DeclareMathOperator{\Sing}{Sing}

\DeclareMathOperator{\inv}{inv}
\DeclarePairedDelimiter{\abs}{\lvert}{\rvert}

\newcommand{\Rmnum}[1]{\MakeUppercase{\romannumeral #1}}

\makeatletter
\@namedef{subjclassname@2020}{\textup{2020} Mathematics Subject Classification}
\makeatother

\title[Equivariant K\"ahler model]
{Equivariant K\"ahler model for Fujiki's class}

\author{Jia Jia}
\author{Sheng Meng}

\address{
	\textsc{National University of Singapore,
		Singapore 119076, Republic of Singapore}
}
\email{jia\_jia@nus.edu.sg}

\address{
	\textsc{Korea Institute For Advanced Study,
		Seoul 02455, Republic of Korea}
}
\email{ms@nus.edu.sg, shengmeng@kias.re.kr}



\begin{document}

\begin{abstract}
	Let $X$ be a compact complex manifold in Fujiki's class $\mathcal{C}$,
	i.e., admitting a big $(1,1)$-class $[\alpha]$.
	Consider $\Aut(X)$ the group of biholomorphic automorphisms
	and $\Aut_{[\alpha]}(X)$ the subgroup of automorphisms preserving the class $[\alpha]$ via pullback.
	We show that $X$ admits an $\Aut_{[\alpha]}(X)$-equivariant K\"{a}hler model:
	there is a bimeromorphic holomorphic map $\sigma \colon \widetilde{X}\to X$
	from a K\"{a}hler manifold $\widetilde{X}$ such that $\Aut_{[\alpha]}(X)$ lifts holomorphically via $\sigma$.

	There are several applications.
	We show that $\Aut_{[\alpha]}(X)$ is a Lie group with only finitely many components.
	This generalizes an early result of Lieberman and Fujiki on the K\"{a}hler case.
	We also show that every torsion subgroup of $\Aut(X)$ is almost abelian,
	and $\Aut(X)$ is finite if it is a torsion group.
\end{abstract}

\maketitle
\setcounter{tocdepth}{1}
\tableofcontents

\section{Introduction}

Let $X$ be a compact complex manifold in \emph{Fujiki's class $\mathcal{C}$},
i.e., one of the following three equivalent assumptions is satisfied:
\begin{enumerate}
	\item $X$ is the meromorphic image of a compact K\"ahler manifold;
	\item $X$ is bimeromorphic to a compact K\"ahler manifold;
	\item $X$ admits a big $(1,1)$-class $[\alpha]$.
\end{enumerate}

We refer to \cite{fujiki1978automorphism}*{Definition~1.1 and Lemma~1.1},
\cite{varouchas1989kahler}*{Chapter~\Rmnum{4}, Theorem~5}
and \cite{demailly2004numerical}*{Theorem~0.7}
for the equivalence and some properties of Fujiki's class $\mathcal{C}$.

It is a natural question that can we study the group of biholomorphic automorphisms $\Aut(X)$
via some well-chosen K\"ahler model of $X$?
So this requires us to trace back to the construction of the K\"ahler model.
Indeed, Demailly and Paun showed in the proof of \cite{demailly2004numerical}*{Theorem~3.4} that
if a compact complex manifold $X$ admits a big $(1,1)$-class $[\alpha]$,
then there is a bimeromorphic holomorphic map $\sigma\colon X'\to X$ from a K\"ahler manifold $X'$
obtained by a sequence of blowups along smooth centres
determined by the ideal sheaf $\mathcal{J}$ corresponding to some
K\"ahler current $T$ with analytic singularities (cf.~\cref{def-sing}) in $[\alpha]$.
In general, $\Aut(X)$ does not lift via $\sigma$,
for the first very simple reason that the class $[\alpha]$ may not be preserved by $\Aut(X)$.
Then we focus ourselves on the subgroup
\[
	\Aut_{[\alpha]}(X)\coloneqq \{g\in \Aut(X)\mid g^*[\alpha]=[\alpha]\}.
\]
However, one still cannot expect the lifting of $\Aut_{[\alpha]}(X)$,
for the second reason that the blown-up ideal sheaf $\mathcal{J}$ is not $\Aut_{[\alpha]}(X)$-invariant.
Therefore, we need to find another K\"ahler model in a more natural way.
Our idea is to consider the ideal sheaf generated by $g^*\mathcal{J}$ for all $g\in \Aut_{[\alpha]}(X)$.

The following is our main result.
Note that the lift action of $\Aut_{[\alpha]}(X)$ on $\widetilde{X}$
is given by $g(-)=(\sigma^{-1}\circ g\circ\sigma)(-)$.
By a holomorphic lifting we mean that the lift action is also holomorphic,
i.e., the induced map $\Aut_{[\alpha]}(X)\times \widetilde{X}\to \widetilde{X}$
is holomorphic.

\begin{theorem}\label{main-thm}
	Let $X$ be a compact complex manifold in Fujiki's class $\mathcal{C}$.
	For any big $(1,1)$-class $[\alpha]$ on $X$,
	there exists a bimeromorphic holomorphic map $\sigma\colon\widetilde{X}\to X$
	from a K\"ahler manifold $\widetilde{X}$
	such that $\Aut_{[\alpha]}(X)$ lifts holomorphically via $\sigma$.
\end{theorem}

\begin{remark}
	The bimeromorphic holomorphic map $\sigma$ in \cref{main-thm} is indeed obtained by
	a sequence of $\Aut_{[\alpha]}(X)$-equivariant blowups along smooth centres.
\end{remark}

Let $X$ be a compact complex manifold.
It is known that $\Aut(X)$ has the natural structure of a complex Lie group
acting biholomorphically on $X$ (cf.~\cite{douady1966probleme}).
Denote by $\Aut_0(X)$ the connected component of $\Aut(X)$ containing the identity.
Since it is connected, the pullback action of $\Aut_0(X)$ on the (discrete) lattice $H^2(X,\mathbb{Z})$ is trivial.
When the $\partial\overline{\partial}$-lemma holds on $X$,
we have the Hodge decomposition that $H^{1,1}(X,\mathbb{R})$ is a subspace of $H^2(X,\mathbb{R})$
and then $\Aut_0(X)$ acts also trivially on $H^{1,1}(X,\mathbb{R})$.
Note that $\partial\overline{\partial}$-lemma holds when $X$ admits a big $(1,1)$-class.
So in this paper, our $\Aut_0(X)$ is always a subgroup of $\Aut_{[\alpha]}(X)$ for any big $(1,1)$-class $[\alpha] \in H^{1,1}(X,\mathbb{R})$.
We refer to \cite{deligne1975real}*{Lemma~(5.15) and Proposition~(5.17)} and \cite{fujiki1978automorphism}*{Proposition~1.6 and Corollary~1.7} for the details.

When $X$ is a K\"ahler manifold with a K\"ahler form $\alpha$,
Lieberman (cf.~\cite{lieberman1978compactness}*{Proposition~2.2}) and Fujiki (cf.~\cite{fujiki1978automorphism}*{Theorem~4.8}), separately, proved that
\[
	[\Aut_{[\alpha]}(X):\Aut_0(X)]<\infty.
\]
Their proof heavily relies on the K\"ahler form $\alpha$ (or at least the existence of a K\"ahler form).

Nevertheless, with the help of our \cref{main-thm}, we can show the following result.

\begin{corollary}\label{main-cor-components}
	Let $X$ be a compact complex manifold (in Fujiki's class $\mathcal{C}$).
	Then
	\[
		[\Aut_{[\alpha]}(X):\Aut_0(X)]<\infty
	\]
	for any big $(1,1)$-class $[\alpha]$ on $X$.
\end{corollary}

We also give applications on torsion group actions.
When $X$ is a projective variety defined over any field $k$ of characteristic $0$,
Javanpeykar \cite{javanpeykar2021arithmetic}*{Theorem~1.4} showed that the group of $k$-automorphisms $\Aut_k(X)$ is finite if it is torsion.
We show that the same result holds true for normal compact complex spaces in Fujiki's class $\mathcal{C}$.

\begin{corollary}\label{main-cor-torsion}
	Let $X$ be a normal compact complex space in Fujiki's class $\mathcal{C}$.
	Then $\Aut(X)/\Aut_0(X)$ has bounded torsion subgroups,
	i.e., there is a constant $C$ such that $|G|\leq C$ for any torsion subgroup $G\leq \Aut(X)/\Aut_0(X)$.
	In particular, $\Aut(X)$ is a torsion group if and only if it is finite.
\end{corollary}

In the previous work \cite{meng2020jordan}, Perroni, Zhang and the second author showed the Jordan property for $\Aut(X)$ when $X$ is a compact complex space in Fujiki's class $\mathcal{C}$.
The proof there developed a trick by finding some invariant K\"ahler submanifold $Z$ (with $\Aut(X)$ shrunk a bit)
and transferring the attention to a new compact K\"ahler manifold $X'$:
the compactified normal bundle $\mathbb{P}_Z(\mathcal{N}_{Z/X}\oplus \mathcal{O}_Z)$.
Note however $\Aut(X')$ only keeps tracking of finite subgroups (more generally reductive subgroups or torsion subgroups) of $\Aut(X)$.

Now with the help of \cref{main-cor-components}, we can provide an alternative proof.

\begin{corollary}\label{main-cor-abelian}
	Let $X$ be a compact complex space in Fujiki's class $\mathcal{C}$.
	Then there exists a constant $J$ such that any torsion subgroup $G$ of $\Aut(X)$ has an abelian subgroup $H\leq G$ with $[G:H]\leq J$.
	In particular, $\Aut(X)$ has Jordan property.
\end{corollary}

\par \vskip 1pc \noindent
\textbf{Acknowledgement.}

The authors would like to thank Professor De-Qi Zhang for many inspiring discussions and valuable suggestions to improve the paper.
The second author would like to thank Professor Fabio Perroni for the invitation of the talk in Universit\`a degli studi di Trieste on December 2021,
whence a rough idea of the main theorem is formulated.
The first author is supported by a President's Scholarship of NUS.
The second author is supported by a Research Fellowship of KIAS (MG075501).

\section{Preliminaries}

Let $X$ be a compact complex manifold with a fixed positive definite Hermitian form $\omega$.
Let $\alpha$ be a closed $(1,1)$-form.
We use $[\alpha]$ to represent its class in $H^{1,1}(X,\mathbb{R})$.
We define the following positivity notions (independent of the choice of $\omega$):
\begin{itemize}
	\item $[\alpha]$ is \emph{K\"ahler} if it contains a K\"ahler form,
	      i.e., if there is a smooth function $\varphi$ such that $\alpha+\frac{\sqrt{-1}}{2\pi}\partial\overline{\partial}\varphi\geq \epsilon\omega$ on $X$ for some $\epsilon>0$.
	\item $[\alpha]$ is \emph{big} if it contains a K\"ahler current $T$, i.e., if there is a quasi-plurisubharmonic function $\varphi\colon X\to \mathbb{R}\cup\{-\infty\}$ such that
	      $T\coloneqq \alpha+\frac{\sqrt{-1}}{2\pi}\partial\overline{\partial}\varphi\geq \epsilon\omega$ holds weakly as currents on $X$ for some $\epsilon>0$.
\end{itemize}
Recall that a \emph{quasi-plurisubharmonic} function means that locally it is given by the sum of a plurisubharmonic function plus a smooth function.

We recall the definition of the non-K\"ahler locus of a big class $[\alpha]$.

\begin{definition}
	Let $X$ be a compact complex manifold and $[\alpha]$ a big $(1,1)$-class.
	Then the \emph{non-K\"ahler locus} of $[\alpha]$ (or $\alpha$) is defined and denoted by
	\[
		E_{nK}(\alpha)\coloneqq E_{nK}([\alpha])\coloneqq \bigcap\limits_{T\in[\alpha]} \Sing(T),
	\]
	where the intersection ranges over all K\"ahler currents $T =\alpha+\sqrt{-1}\partial\bar{\partial}\varphi$ in the class $[\alpha]$,
	and $\Sing(T)$ is the complement of the set of points $x\in X$ such that $\varphi$ is smooth near $x$.
\end{definition}

We recall the basic definition of analytic singularities (cf.~\cite{boucksom2002volume}*{Section~2.1}).
Note that the data $(\mathcal{J},c)$ below is not uniquely determined by the function $\varphi$ with analytic singularities.

\begin{definition}\label{def-sing}
	Let $X$ be a compact complex manifold and $[\alpha]$ a closed $(1,1)$-class.
	\begin{enumerate}
		\item Given a coherent ideal sheaf $\mathcal{J}$ and a constant $c>0$,
		      we say that a function $\varphi$ has \emph{singularities of type $(\mathcal{J},c)$} if locally it can be written as
		      \[
			      \varphi=c\log\Big(\sum_{j=1}^n \abs{f_j}^2\Big)+h
		      \]
		      for some local generators $(f_j)$ of $\mathcal{J}$ and some smooth function $h$.
		\item We say that $\varphi$ has \emph{analytic singularities} if it has singularities of type $(\mathcal{J},c)$
		      for some coherent ideal sheaf $\mathcal{J}$ and some constant $c>0$.
		\item We also say that a closed current $T\in [\alpha]$ has \emph{analytic singularities}
		      if it can be written as $T=\alpha+\frac{\sqrt{-1}}{2\pi}\partial\overline{\partial}\varphi$ such that the potential function $\varphi$ has analytic singularities.
	\end{enumerate}
\end{definition}

The following is a direct application of the regularization theorem by Demailly~\cite{demailly1992regularization};
see also \cite{demailly2004numerical}*{Theorem~3.2}.

\begin{theorem}\label{thm-analytic}
	Let $X$ be a compact complex manifold with a big $(1,1)$-class $[\alpha]$.
	Then there exists a K\"ahler current $T\in [\alpha]$ with analytic singularities.
\end{theorem}

Boucksom~\cite{boucksom2004divisorial}*{Theorem~3.17} observed further that indeed one can find a K\"ahler current with analytic singularities and also with minimal singular locus.

\begin{theorem}\label{thm-boucksom}
	Let $X$ be a compact complex manifold with a big $(1,1)$-class $[\alpha]$.
	Then there exists a K\"ahler current $T\in [\alpha]$ with analytic singularities such that $\Sing(T)=E_{nK}(\alpha)$.
	In particular, $E_{nK}(\alpha)$ is a closed analytic subspace of $X$ and $E_{nK}(\alpha)=\emptyset$ if and only if $[\alpha]$ is K\"ahler.
\end{theorem}

\section{Equivariant K\"ahler model}

In this section, we prove \cref{main-thm}.

First, we recall the following theorem on the equivariant log resolution of an ideal sheaf by Bierstone and Milman~\cite{bierstone1997canonical}*{Theorem~1.10}.
\begin{theorem}\label{thm-BM}
	Let $X$ be a compact complex manifold with a coherent ideal sheaf $\mathcal{J}$.
	Let $G\leq \Aut(X)$ such that $g^*\mathcal{J}=\mathcal{J}$ for any $g\in G$.
	Then there is a finite sequence
	\[
		\xymatrix{
		X_k \ar[r]^{\sigma_k} & \cdots \ar[r]^{\sigma_2} & X_1 \ar[r]^-{\sigma_1} & X_0 = X \\
		}
	\]
	of $G$-equivariant blowups $\sigma_j$, $j = 1,\dots,k$,
	along smooth centres,
	such that $\sigma^{-1}\mathcal{J}\cdot \mathcal{O}_{X_k}$ is a normal-crossings divisor.
\end{theorem}

\begin{proof}
	In \cite{bierstone1997canonical}*{Theorem~1.10}, the sequence is taken as blowups along $\inv_{\mathcal{J}}$-admissible centres.
	Note that an $\inv_{\mathcal{J}}$-admissible centre is determined by $\mathcal{J}$ itself and hence $G$-invariant.
	So the blowups are $G$-equivariant.
	We refer to \cite{wlodarczyk2009resolution} for the details.
\end{proof}

The following lemma is well-known, and we give a proof for the convenience of the readers.
The result holds true in the algebraic setting with the same proof.

\begin{lemma}\label{lem-hol-lift}
	Let $\sigma\colon X'\to X$ be a bimeromorphic holomorphic map of compact complex spaces.
	Let $G$ be a complex Lie group acting holomorphically on $X$ such that $G$ lifts via $\sigma$.
	Then the induced action of $G$ on $X'$ is also holomorphic.
\end{lemma}

\begin{proof}
	% We may assume $X$ is irreducible.
	% Let $E$ be the exceptional locus of $\sigma$ and $U'\coloneqq X'\backslash E$.
	% Then $\sigma|_{U'} \colon U'\to U\coloneqq \sigma(U')$ is isomorphic and $U$ and $U'$ are Zariski open dense in $X$ and $X'$, respectively.
	% Note that $U$ and $U'$ are $G$-invariant since $G$ preserves fibres of $\sigma$ and their dimension.

	Consider the holomorphic map
	\[
		\phi \colon G\times X'\times X'\to G\times X\times X
	\]
	via $(g,x',y')\mapsto (g,\sigma(x'),\sigma(y'))$.
	% Note that $\phi \colon G\times U'\times U'\to G\times U\times U$ is isomorphic.

	Consider the graph of $G$ on $X$:
	\[
		\Gamma\coloneqq \{(g,x,y)\in G\times X\times X\mid y=g(x)\}
	\]
	which is a closed analytic subspace of $G\times X\times X$ since the $G$-action on $X$ is holomorphic.

	For each $g\in G$,
	let $\Gamma_{g|_{X}}$ be the graph of $g|_X$ viewed as a fibre of $\Gamma\to G$ over $g$.
	Define $\Gamma_{g|_{X'}}$ similarly.
	Then $\Gamma_{g|_{X'}}$ is the proper transform of $\Gamma_{g|_X}$ in $G\times X'\times X'$ via $\phi$.
	Let $\Gamma'$ be the proper transform of $\Gamma$ in $G\times X'\times X'$.
	Then $\Gamma'$ is a closed analytic subspace of $G\times X'\times X'$.
	By taking the first projection, the fibre of $\Gamma'\to G$ over $g$ is just $\Gamma_{g|_{X'}}\cong X'$ since $g \colon X'\to X'$ is biholomorphic.
	Therefore, the projection to the first two factors $\Gamma'\to G\times X'$ is biholomorphic.
	Note that $\Gamma'$ is the graph of $G$ on $X'$.
	So the $G$-action on $X'$ is holomorphic.
\end{proof}

\begin{proof}[Proof of ~\cref{main-thm}]
	By \cref{thm-analytic},
	there exists a K\"ahler current
	\[
		T=\alpha+\frac{\sqrt{-1}}{2\pi}\partial\overline{\partial}\varphi \in [\alpha]
	\]
	with analytic singularities of type $(\mathcal{J}_{\varphi},c)$
	for some coherent ideal sheaf $\mathcal{J}_{\varphi}$ and $c>0$.
	This means that we may write (locally) that
	\[
		\varphi=c\log \Big(\sum_{j=1}^n \abs{f_j}^2\Big)+h
	\]
	where $(f_j)$ are local generators of $\mathcal{J}_{\varphi}$ and $h$ is smooth.
	Consider the $\Aut_{[\alpha]}(X)$-invariant ideal sheaf
	\[
		\mathcal{J}\coloneqq \sum_{g\in \Aut_{[\alpha]}(X)}g^*\mathcal{J}_{\varphi}.
	\]
	Since $X$ is compact, $\mathcal{J}$ is also coherent (cf.~\cite{demailly1997complex}*{Chapter~\Rmnum{2}, Property~(3.22)}).
	Then we may write
	\[
		\mathcal{J}=\sum_{i=1}^m \mathcal{J}_i
	\]
	where $\mathcal{J}_i\coloneqq g_i^*\mathcal{J}_{\varphi}$ for some $g_i\in \Aut_{[\alpha]}(X)$.

	By \cref{thm-BM}, there is an $\Aut_{[\alpha]}(X)$-equivariant biholomorphic holomorphic map
	\[
		\sigma \colon \widetilde{X}\to X
	\]
	which is obtained by a sequence of blowups
	along smooth centres such that $\sigma^{-1}\mathcal{J}\cdot \mathcal{O}_{\widetilde{X}}$ is invertible.

	Let $\mathcal{J}'_i\coloneqq \sigma^{-1}\mathcal{J}_i\cdot \mathcal{O}_{\widetilde{X}}$.
	Then
	\[
		\sum_{i=1}^m \mathcal{J}'_i=\sigma^{-1}\mathcal{J}\cdot \mathcal{O}_{\widetilde{X}}
	\]
	is invertible.
	% Denote by the common divisorial part
	% \[
	% 	\mathcal{J}'_{gcd}\coloneqq gcd(\mathcal{J}'_1,\cdots, \mathcal{J}'_n).
	% \]
	% Denote by $\widetilde{\mathcal{J}}_i\coloneqq \mathcal{J}'_i\cdot (\mathcal{J}'_{gcd})^{-1}$.
	Denote by
	\[
		\widetilde{\mathcal{J}}_i\coloneqq \mathcal{J}'_i\cdot \big(\sigma^{-1}\mathcal{J}\cdot \mathcal{O}_{\widetilde{X}}\big)^{-1}
	\]
	which is still an ideal sheaf.
	Then
	\[
		\sum_{i=1}^m \widetilde{\mathcal{J}}_i =\mathcal{O}_{\widetilde{X}}.
	\]
	We give several more notations.
	\begin{itemize}[leftmargin=2em]
		\item Let $s$ be the local generator of invertible sheaf $\sigma^{-1}\mathcal{J}\cdot \mathcal{O}_{\widetilde{X}}$.
		\item We define a positive current $D$ which is the integration current along the divisor $\{s=0\}$.
		      By the Poincar\'e-Lelong formula,
		      $D$ can be locally written as $\frac{\sqrt{-1}}{2\pi}\partial\overline{\partial}\log \abs{s}^2$.
		\item Let $E$ be the reduced (full) exceptional divisor (locus) of $\sigma$.
		\item Fix a (positive definite) Hermitian form $\omega$ on $X$ such that $g_i^*T\geq \omega$ for each $i$.
		\item Let $u_E$ be some closed smooth $(1,1)$-form in the class $[E]$ such that $\sigma^*\omega-\epsilon u_E$ is a positive definite Hermitian form when $\epsilon>0$ is sufficiently small (cf.~\cite{demailly2004numerical}*{Proof of Lemma~3.5}).
	\end{itemize}

	We show that $\widetilde{X}$ is K\"ahler by the following claim.
	\begin{claim}\label{claim}
		The class $[\widetilde{\alpha}]\coloneqq [\sigma^*\alpha-c D-\epsilon u_E]$ contains a K\"ahler form for sufficiently small $\epsilon>0$.
	\end{claim}

	Let $T_i'\coloneqq \sigma^*g_i^*T$ and $f_{j,i}'\coloneqq f_j\circ g_i\circ \sigma$.
	Then the potential function of $T_i'$ is locally of the form
	\[
		\varphi_i'=c\log \Big(\sum_{j=1}^n \abs{f_{j,i}'}^2\Big)+h\circ g_i\circ \sigma
	\]
	and $(f_{j,i}')_j$ are local generators of $\mathcal{J}'_i$.
	Let $s_i$ be the g.c.d.\ of the $(f_{j,i}')$'s.
	Then we can write down the Siu's decomposition
	\[
		T_i'=R_i+c D_i
	\]
	where $D_i$ is the integration current along the divisor $\{s_i=0\}$
	which can be locally written as $\frac{\sqrt{-1}}{2\pi}\partial\overline{\partial}\log \abs{s_i}^2$
	and $R_i\geq \sigma^*\omega$ such that the Lelong super-level (analytic) set $E_c(R_i)$ has codimension at least $2$
	(cf.~\cite{boucksom2004divisorial}*{Section~2.2.1--2.2.2} and \cite{boucksom2002volume}*{Section~2.2}).
	Note that $s$ is a factor of $s_i$ since $\mathcal{J}'_i\subseteq \sigma^{-1}\mathcal{J}\cdot\mathcal{O}_{\widetilde{X}}$.
	Then we have $D_i\geq D$ and hence
	\[
		T'_i-c D\geq R_i\geq \sigma^*\omega.
	\]

	We now construct
	\[
		\widetilde{T}_i\coloneqq T_{i}' - cD - \epsilon u_E=(\sigma^*g_{i}^{*}\alpha-\epsilon u_E)+\frac{\sqrt{-1}}{2\pi}\partial \overline{\partial}\varphi_{i}'-cD\in [\widetilde{\alpha}]
	\]
	with $\epsilon>0$ sufficiently small.
	By the construction, the potential function of $\widetilde{T}_i$ can be locally written as
	\begin{align*}
		\widetilde{\varphi}_i & = c\log \Big(\sum_{j=1}^n \abs{f_{j,i}'}^2\Big) + h\circ g_i\circ \sigma-c\log\abs{s}^2 \\
		                      & = c\log \Big(\sum_{j=1}^n \abs{\widetilde{f}_{j,i}}^2\Big)+h\circ g_i\circ\sigma
	\end{align*}
	where $f_{j,i}'=\widetilde{f}_{j,i}\cdot s$.
	Note that
	\[
		\widetilde{T}_i\geq \sigma^*\omega-\epsilon u_E
	\]
	while the right-hand side is positive definite by the construction before \cref{claim}.
	Therefore, $\widetilde{T}_i$ is also a K\"ahler current.

	Note that $\widetilde{\mathcal{J}}_i$ is locally generated by $(\widetilde{f}_{j,i})_j$.
	Then
	\[
		\bigcap_{i=1}^m \Sing(\widetilde{T}_i)=\bigcap_{i=1}^m V(\widetilde{\mathcal{J}}_i)=V\Big(\sum_{i=1}^m \widetilde{\mathcal{J}}_i\Big)=\emptyset
	\]
	where $V(-)$ is the zeros of the ideal sheaf.
	In particular, $E_{nK}(\widetilde{\alpha})=\emptyset$ and the claim is proved by \cref{thm-boucksom}.

	Finally, note that the induced action of $\Aut_{[\alpha]}(X)$ on $\widetilde{X}$ is holomorphic by \cref{lem-hol-lift}.
\end{proof}

\section{Boundedness on group components}

In this section, we are going to prove \cref{main-cor-components,main-cor-abelian}.
Our \cref{main-cor-components} plays a key role in the reduction to the K\"ahler case.

Recall that Lieberman~\cite{lieberman1978compactness}*{Proposition 2.2} and Fujiki~\cite{fujiki1978automorphism}*{Theorem 4.8} showed separately that
\[
	[\Aut_{[\alpha]}(X):\Aut_0(X)]<\infty
\]
for a compact K\"ahler manifold $X$ and a K\"ahler form $\alpha$.
For the convenience of our proof for \cref{main-cor-components},
we give a version on a big class generalized by Dinh, Hu and Zhang;
see \cite{dinh2015compact}*{Theorem~2.1} for a more generalized setting.
Note however their proof still requires the existence of a K\"ahler form.
We refer to \cite{meng2018building}*{Propositions~2.9 and 3.6} for a generalized explanation by cone analysis and linear algebra.

\begin{theorem}\label{thm-dhz}
	Let $X$ be a compact K\"ahler manifold.
	Then
	\[
		[\Aut_{[\alpha]}(X):\Aut_0(X)]<\infty
	\]
	for any big $(1,1)$-class $[\alpha]$.
\end{theorem}

We need the following lemma concerning the descending of
connected complex Lie group action via a bimeromorphic holomorphic map.

\begin{lemma}\label{lem-rig}
	Let $\sigma \colon X'\to X$ be a bimeromorphic holomorphic map of normal compact complex spaces.
	Then the map
	\[
		\tau\colon \Aut_0(X')\to \Aut_0(X),\quad g\mapsto \sigma\circ g\circ\sigma^{-1}
	\]
	is an injective complex Lie group homomorphism.
\end{lemma}

\begin{proof}
	Note that $\sigma_*\mathcal{O}_{X'}=\mathcal{O}_X$ and $\Aut_0(X')$ is connected.
	Then the lemma is essentially a corollary of the rigidity lemma (cf.~\cite{akhiezer1995lie}*{\S~2.4, Lemmas~1 and 2}).
	Note that $\tau$ is injective because $\sigma$ is bimeromorphic.
\end{proof}

\begin{theorem}\label{thm-aut0}
	Let $X$ be a normal compact complex space in Fujiki's class $\mathcal{C}$.
	Then there is a K\"ahler model $\sigma \colon \widetilde{X}\to X$ such that the map
	\[
		\tau \colon \Aut_0(\widetilde{X})\to \Aut_0(X)
	\]
	via $\tau(g)\coloneqq \sigma\circ g\circ \sigma^{-1}$ is a complex Lie group isomorphism.
\end{theorem}

\begin{proof}
	Let $\pi\colon X'\to X$ be an $\Aut(X)$-equivariant resolution of singularities (cf.~\cite{bierstone1997canonical}*{Theorem~13.2}).
	It follows from Lemmas \ref{lem-hol-lift} and \ref{lem-rig} that the map
	\[
		\Aut_0(X')\to \Aut_0(X),\quad g\mapsto \pi\circ g\circ\pi^{-1}
	\]
	is a complex Lie group isomorphism since $X$ is normal (see also \cite{fujiki1978automorphism}*{Lemma~2.5}).
	Note that $X'$ is also in Fujiki's class $\mathcal{C}$.
	So we may replace $X$ with $X'$ and assume that $X$ is smooth.

	Since $\Aut_0(X)$ is connected,
	its pullback action on $H^2(X,\mathbb{Z})$ is trivial.
	Note that the $\partial\bar{\partial}$-lemma holds for compact complex manifolds in Fujiki's class $\mathcal{C}$.
	So $H^{1,1}(X,\mathbb{R})$ is a subspace of $H^2(X,\mathbb{R})$ and $\Aut_0(X)$ acts trivially on $H^{1,1}(X,\mathbb{R})$.
	Let $[\alpha]\in H^{1,1}(X,\mathbb{R})$ be a big $(1,1)$-class.
	Note that $\Aut_0(X)\leq \Aut_{[\alpha]}(X)$.

	We take $\sigma$ as in \cref{main-thm}.
	Then the map
	\[
		\phi \colon \Aut_0(X)\times \widetilde{X}\to \widetilde{X}
	\]
	via $(g,\widetilde{x})\mapsto (\sigma^{-1}\circ g\circ \sigma)(\widetilde{x})$
	is well-defined and holomorphic.
	In particular, $\Aut_0(X)$ lifts to a (unique) subgroup of $\Aut_0(\widetilde{X})$.
	By \cref{lem-rig}, the map
	\[
		\tau \colon \Aut_0(\widetilde{X})\to \Aut_0(X)
	\]
	via $\tau(g)\coloneqq \sigma\circ g\circ \sigma^{-1}$ is an injective complex Lie group homomorphism.
	We just see the surjectivity of $\tau$ by the lifting property.
	So $\tau$ is isomorphic.
\end{proof}

\begin{proof}[Proof of \cref{main-cor-components}]
	Let $[\alpha]$ be a big $(1,1)$-class.
	By \cref{main-thm}, there is a K\"ahler model $\sigma\colon\widetilde{X}\to X$ such that $\Aut_{[\alpha]}(X)$ lifts to a group $G\leq \Aut(\widetilde{X})$ via $\sigma$.
	By \cref{thm-aut0}, $\Aut_0(\widetilde{X})\leq G$.
	Note that $G\leq \Aut_{\sigma^*[\alpha]}(\widetilde{X})$ and $\sigma^*[\alpha]$ is still big.
	Since $\widetilde{X}$ is K\"ahler, we have that
	\[
		[\Aut_{\sigma^*[\alpha]}(\widetilde{X}):\Aut_0(\widetilde{X})]<\infty
	\]
	by \cref{thm-dhz}.
	Finally, note that
	\[
		\Aut_{[\alpha]}(X)/\Aut_0(X)\cong G/\Aut_0(\widetilde{X})\leq \Aut_{\sigma^*[\alpha]}(\widetilde{X})/\Aut_0(\widetilde{X}).
	\]
	So the corollary is proved.
\end{proof}

\begin{remark}\label{rem-normal-necessary}
	\Cref{lem-rig} and also \cref{thm-aut0} fail in general if $X$ has non-normal singularities such that $\sigma_*\mathcal{O}_{X'}=\mathcal{O}_X$ does not hold.
	A simple example is by taking $X'=\mathbb{P}^1$, $X=\{y^2z-x^3=0\}$ the cuspidal curve, and $\sigma$ just the normalization of $X$.
	Note that $\Aut(X')=\Aut_0(X')=\PGL(2)$ while $\Aut_0(X)$ is (conjugate to) the subgroup of upper triangular matrices in $\PGL(2)$.
\end{remark}

\section{Torsion group actions}

In this section, we will prove \cref{main-cor-torsion,main-cor-abelian}.
Our \cref{main-thm} plays a key role in the reduction to the connected Lie group action.

The following result holds true for algebraic groups defined over an algebraically closed field of characteristic $0$;
see \cite{javanpeykar2021arithmetic}*{Lemma~5.4}.
The proof there cannot be applied well to the case of connected real Lie groups, e.g., we do not have the Chevalley decomposition.
So we give a fundamental Lie-group theoretical proof for the convenience of the readers.
\begin{lemma}\label{lem-torsion-lie}
	A torsion connected real Lie group $G$ is trivial.
\end{lemma}

\begin{proof}
	Consider the adjoint representation
	\[
		\rho\colon G\to \GL(T_{G,e})
	\]
	where $e$ the identity element of $G$ and $T_{G,e}$ is the tangent space of $G$ at $e$.
	Note that $\Ker \rho$ is just the centre of $G$ since $G$ is connected.

	We claim that $\rho(G)$ is trivial.
	Suppose the contrary.
	Then $\rho(G)$ is infinite since it is connected.
	Note that $\rho(G)$ is a torsion connected Lie subgroup of $\GL(T_{G,e})$.
	So any finitely generated subgroup of $\rho(G)$ is finite by Schur's theorem (cf.~\cite{lam2001first}*{Chapter~3, \S~9, Theorem~(9.9)}).
	Then we may choose a sequence of $g_i\in \rho(G)$ such that $g_{n+1}\not\in G_n$ where $G_n$ is the (finite) group generated by $(g_1,\dots, g_n)$.
	Note that the order $|G_n|$ increases strictly.
	Let $M$ be a maximal compact subgroup of $\rho(G)$ which is unique up to conjugation.
	Then $G_n$ can be viewed as a subgroup of $M$ via a conjugation.
	In particular, $M$ is an infinite group.
	Let $M_0$ be the neutral component of $M$ which is also infinite since $M$ is compact.
	Let $T$ be a maximal torus contained in $M_0$.
	Since $T$ is torsion, $T$ is trivial.
	However, every element of $M_0$ is conjugate to an element of $T$ by the torus theorem.
	Then $M_0$ is trivial, a contradiction.
	So the claim is proved.

	Now $\rho(G)$ is trivial and hence $G=\Ker \rho$ is abelian.
%	$g$ acts, via the conjugate action, trivially on $G$ (cf.~\cite{mundetiriera2019finite}*{Lemma~2.1(2)}).
%	In particular, $G$ is abelian.
	Note that an abelian connected real Lie group $G$ is isomorphic to $\mathbb{R}^m\times (S^1)^n$.
	The latter is easily seen to be non-torsion unless $m=n=0$.
	Therefore, $G$ is trivial.
\end{proof}

For general linear groups over number fields, we may even have boundedness on their torsion subgroups.
\begin{lemma}\label{lem-torsion-finite}
	Let $G$ be a torsion subgroup of $\GL_n(K)$ where $K$ is a number field.
	Then $\abs{G}\leq N$ for some constant $N$ depending only on $n$ and $K$.
\end{lemma}
\begin{proof}
	By the Minkowski's~theorem (cf.~\cite{serre2006bounds}*{Theorem~5, and \S~4.3}), there is a constant $M$ depending only on $n$ and $K$ such that the order $o(g)\leq M$ for any $g\in \GL_n(K)$ with finite order.
	By the Burnside's first theorem (cf.~\cite{lam2001first}*{Chapter~3, \S~9, Theorem~(9.4)}), $\abs{G}\leq N \coloneqq M^{n^3}$ is finite.
\end{proof}

\begin{proof}[Proof of \cref{main-cor-torsion}]
	Let $G$ be a torsion subgroup of $\Aut(X)/\Aut_0(X)$.
	Let $\pi\colon X'\to X$ be an $\Aut(X)$-equivariant resolution of singularities (cf.~\cite{bierstone1997canonical}*{Theorem~13.2}),
	with $\Aut(X)$ lifts to a (unique) subgroup of $\Aut(X')$ via $\pi$.
	Note that $\Aut_0(X)$ lifts (isomorphically) to $\Aut_0(X')$ (cf.~\cite{fujiki1978automorphism}*{Lemma~2.5}).
	Then $G$ also lifts to a (unique) torsion subgroup of $\Aut(X')/\Aut_0(X')$.
	Note that $X'$ is also in Fujiki's class $\mathcal{C}$.
	Therefore, we may replace $X$ with $X'$ and assume that $X$ is smooth.

	Since the pullback action $\Aut_0(X)|_{H^2(X,\mathbb{Q})}$ is trivial, we have the following exact sequence
	\[
		1\longrightarrow G_{\tau}\longrightarrow G\longrightarrow G|_{H^2(X,\mathbb{Q})}\longrightarrow 1
	\]
	where $G_{\tau}$ is the kernel.
	Note that $G|_{H^2(X,\mathbb{Q})}$ is torsion and hence finite with order bounded by some $N$ depending only on $H^2(X,\mathbb{Q})$ (and hence only on $X$) by \cref{lem-torsion-finite}.
	Denote by
	\[
		\Aut_{\tau}(X)\coloneqq \{g\in \Aut(X) \mid g^*|_{H^2(X,\mathbb{Q})}=\id\}.
	\]
	Note that
	\[
		G_{\tau}\leq \Aut_{\tau}(X)/\Aut_0(X)\le \Aut_{[\alpha]}(X)/\Aut_0(X)
	\]
	for any big $(1,1)$-class $[\alpha]$.
	By \cref{main-cor-components}, we have
	\[
		|G_{\tau}|\leq C\coloneqq [\Aut_{\tau}(X):\Aut_0(X)],
	\]
	where $C$ depends only on $X$.
	Then $|G|\le N\cdot C$ and we get an upper bound.

	Finally, if $\Aut(X)$ is torsion, then $\Aut_0(X)$ is trivial by \cref{lem-torsion-lie} and hence $\Aut(X)$ is finite.
\end{proof}

\begin{remark}
	Currently, the normality assumption on $X$ is required in the proof of \cref{main-cor-torsion}.
	The reason is that the equivariant resolution of singularities may enlarge $\Aut_0(X)$;
	see \cref{rem-normal-necessary}.
\end{remark}

It is well known that any torsion subgroup of the general linear group is almost abelian by the Jordan--Schur lemma.
Lee generalized it to the case of connected real Lie groups;
see \cite{lee1976torsion}.

\begin{theorem}\label{thm-lee}
	Let $G$ be a connected real Lie group.
	Then any torsion subgroup $H\leq G$ has an abelian subgroup $H'\leq H$ with
	$[H:H']\leq J$ where $J$ is a constant depending only on $G$.
\end{theorem}

\begin{proof}[Proof of \cref{main-cor-abelian}]
	First, by taking an $\Aut(X)$-equivariant resolution of singularities (cf.~\cite{bierstone1997canonical}*{Theorem~13.2}) which is still in Fujiki's class $\mathcal{C}$, 
	we may assume $X$ is smooth.

	Let $G\leq \Aut(X)$ be a torsion subgroup.
	Denote by
	\[
		\Aut_{\tau}(X)\coloneqq \{ g\in \Aut(X) \mid g^*|_{H^2(X,\mathbb{Q})}=\id\}.
	\]
	Note that $G/G\cap \Aut_{\tau}(X)$ can be viewed as a torsion subgroup of $\GL(H^2(X,\mathbb{Q}))$.
	By \cref{lem-torsion-finite}, we have
	\[
		[G:G\cap \Aut_{\tau}(X)]\leq N
	\]
	for some constant $N$ depending only on $H^2(X,\mathbb{Q})$ (and hence only on $X$).

	Since $\Aut_{\tau}(X)\subseteq\Aut_{[\alpha]}(X)$ and by \cref{main-cor-components},
	we see that
	\[
		C\coloneqq [\Aut_{\tau}(X):\Aut_0(X)]<\infty
	\]
	and hence
	\[
		[G:G\cap \Aut_0(X)]\leq N\cdot C.
	\]
	By \cref{thm-lee}, there is an abelian subgroup $H\leq G\cap \Aut_0(X)$ such that
	\[
		[G\cap \Aut_0(X):H]\leq J_0
	\]
	where $J_0$ is a constant depending only on $\Aut_0(X)$ (and hence only on $X$).

	Together, we have
	\[
		[G:H]\leq J\coloneqq N\cdot C\cdot J_0
	\]
	as desired.
\end{proof}


\begin{bibdiv}
	\begin{biblist}

		\bib{akhiezer1995lie}{book}{
			author={Akhiezer, Dmitri~N.},
			editor={Diederich, Klas},
			title={Lie group actions in complex analysis},
			series={Aspects of {{Mathematics}}},
			publisher={Friedr. Vieweg \& Sohn, Braunschweig},
			address={{Wiesbaden}},
			date={1995},
			volume={27},
			pages={viii+201},
			isbn={978-3-322-80269-9 978-3-322-80267-5},
		}

		\bib{bierstone1997canonical}{article}{
			author={Bierstone, Edward},
			author={Milman, Pierre~D.},
			title={Canonical desingularization in characteristic zero by blowing up the maximum strata of a local invariant},
			date={1997},
			ISSN={0020-9910, 1432-1297},
			journal={Inventiones Mathematicae},
			volume={128},
			number={2},
			pages={207\ndash 302},
		}

		\bib{boucksom2002volume}{article}{
			author={Boucksom, S{\'e}bastien},
			title={On the volume of a line bundle},
			date={2002},
			ISSN={0129-167X, 1793-6519},
			journal={Int. J. Math.},
			volume={13},
			number={10},
			pages={1043\ndash 1063},
		}

		\bib{boucksom2004divisorial}{article}{
			author={Boucksom, S{\'e}bastien},
			title={Divisorial {{Zariski}} decompositions on compact complex manifolds},
			date={2004},
			ISSN={00129593},
			journal={Ann. Sci. \'Ecole Norm. Sup. (4)},
			volume={37},
			number={1},
			pages={45\ndash 76},
		}

		\bib{demailly1992regularization}{article}{
			author={Demailly, Jean-Pierre},
			title={Regularization of closed positive currents and intersection theory},
			date={1992},
			ISSN={2300-7443},
			journal={J. Algebraic Geom.},
			number={1},
			pages={361\ndash 409},
		}

		\bib{demailly1997complex}{book}{
			author={Demailly, Jean-Pierre},
			title={Complex analytic and differential geometry},
			publisher={{Universit\'e de Grenoble I}},
			date={1997},
		}

		\bib{deligne1975real}{article}{
			author={Deligne, Pierre},
			author={Griffiths, Phillip},
			author={Morgan, John},
			author={Sullivan, Dennis},
			title={Real homotopy theory of k\"ahler manifolds},
			date={1975},
			journal={Invent. Math.},
			volume={29},
			pages={245\ndash 274},
		}

		\bib{dinh2015compact}{article}{
			author={Dinh, Tien-Cuong},
			author={Hu, Fei},
			author={Zhang, De-Qi},
			title={Compact {{K\"ahler}} manifolds admitting large solvable groups of automorphisms},
			date={2015},
			ISSN={00018708},
			journal={Advances in Mathematics},
			volume={281},
			pages={333\ndash 352},
		}

		\bib{douady1966probleme}{article}{
			author={Douady, Adrien},
			title={{Le probl\`eme des modules pour les sous-espaces analytiques compacts d'un espace analytique donn\'e}},
			date={1966},
			ISSN={0373-0956},
			journal={Ann. Inst. Fourier (Grenoble)},
			volume={16},
			number={1},
			pages={1\ndash 95},
		}

		\bib{demailly2004numerical}{article}{
			author={Demailly, Jean-Pierre},
			author={Paun, Mihai},
			title={Numerical characterization of the {{K\"ahler}} cone of a compact {{K\"ahler}} manifold},
			date={2004},
			ISSN={0003-486X},
			journal={Ann. Math.},
			volume={159},
			number={3},
			pages={1247\ndash 1274},
		}

		\bib{fujiki1978automorphism}{article}{
			author={Fujiki, Akira},
			title={On automorphism groups of compact {{K\"ahler}} manifolds},
			date={1978},
			journal={Invent. Math.},
			volume={44},
			pages={225\ndash 258},
		}

		\bib{javanpeykar2021arithmetic}{article}{
			author={Javanpeykar, Ariyan},
			title={Arithmetic hyperbolicity: automorphisms and persistence},
			date={2021},
			ISSN={0025-5831, 1432-1807},
			journal={Math. Ann.},
			volume={381},
			number={1-2},
			pages={439\ndash 457},
		}

		\bib{lam2001first}{book}{
			author={Lam, Tsit-Yuen},
			title={A first course in noncommutative rings},
			series={Graduate {{Texts}} in {{Mathematics}}},
			publisher={{Springer}},
			address={{New York}},
			date={2001},
			volume={131},
			ISBN={978-0-387-95325-0 978-1-4419-8616-0},
		}

		\bib{lee1976torsion}{article}{
			author={Lee, Dong~Hoon},
			title={On torsion subgroups of {{Lie}} groups},
			date={1976},
			journal={Proc. Amer. Math. Soc.},
			volume={55},
			pages={3},
		}

		\bib{lieberman1978compactness}{incollection}{
			author={Lieberman, David~I.},
			title={Compactness of the {{Chow}} scheme: applications to automorphisms and deformations of {{Kahler}} manifolds},
			date={1978},
			booktitle={Fonctions de {{Plusieurs Variables Complexes III}}},
			editor={Norguet, Fran{\c c}ois},
			volume={670},
			publisher={{Springer}},
			address={{Berlin, Heidelberg}},
			pages={140\ndash 186},
		}

		\bib{meng2020jordan}{article}{
			author={Meng, Sheng},
			author={Perroni, Fabio},
			author={Zhang, De-Qi},
			title={Jordan property for automorphism groups of compact spaces in {{Fujiki}}'s class {{C}}},
			date={2020},
			journal={arXiv},
			eprint={2011.09381},
		}


		\bib{meng2018building}{article}{
			author={Meng, Sheng},
			author={Zhang, De-Qi},
			title={Building blocks of polarized endomorphisms of normal projective varieties},
			date={2018},
			volume={325},
			journal={Adv. Math.},
			pages={243\ndash 273},
		}

		\bib{serre2006bounds}{incollection}{
			author={Serre, Jean-Pierre},
			title={Bounds for the orders of the finite subgroups of {{G}}(k)},
			date={2006},
			booktitle={Group representation theory},
			publisher={{EPFL Press}},
			address={{Lausanne}},
			pages={405\ndash 450},
		}

		\bib{varouchas1989kahler}{article}{
			author={Varouchas, Jean},
			title={K\"ahler spaces and proper open morphisms},
			date={1989},
			ISSN={0025-5831, 1432-1807},
			journal={Math. Ann.},
			volume={283},
			number={1},
			pages={13\ndash 52},
		}

		\bib{wlodarczyk2009resolution}{inproceedings}{
			author={Wlodarczyk, Jaroslaw},
			title={Resolution of singularities of analytic spaces},
			date={2009},
			booktitle={Proceedings of {{G\"okova Geometry-Topology Conference}}},
			pages={31\ndash 63},
		}

	\end{biblist}
\end{bibdiv}
\end{document}